\documentclass[12 pt, reqno]{amsart}
\usepackage{style,fullpage}

\begin{document}

\title{$K$-Theory of Multiparameter Persistence Modules: Additivity}


\author{Ryan Grady
\and Anna Schenfisch}

\address{Department of Mathematical Sciences\\Montana State University\\Bozeman, MT 59717}
\email{ryan.grady1@montana.edu}

\address{Department of Mathematical Sciences\\Montana State University\\Bozeman, MT 59717}
\email{annaschenfisch@montana.edu}



\maketitle

\tableofcontents
\begin{abstract}
Persistence modules stratify their underlying parameter space, a quality that make persistence modules amenable to study via invariants of stratified spaces.
In this article, we extend a result previously known only for one-parameter persistence modules to grid multi-parameter persistence modules. Namely, we show the $K$-theory of grid multi-parameter persistence modules is additive over strata. This is true for both standard monotone multi-parameter persistence as well as multi-parameter notions of zig-zag persistence. We compare our calculations for the specific group $K_0$ with the recent work of Botnan, Oppermann, and Oudot, highlighting and explaining the differences between our results through an explicit projection map between computed groups.
\end{abstract}
\section{Introduction}

%

Persistent homology has become a central tool in {\em topological data analysis} (TDA). It has also been added to the toolbox of symplectic geometers, see for instance \cite{PoltBook}. The typical setting for persistent homology is that of a filtered topological space, $Y_\bullet$, and a field of coefficients, $\FF$. The sequence of vector spaces and linear maps encoded by $H_\ast (Y_\bullet; \FF)$ are then used to analyze the space $Y_\bullet$ and/or a dataset from which $Y_\bullet$ has been constructed. Persistence modules are a generalization of persistent homology in that they are simply a functorial assignment---say from filtered spaces---to a ``reasonable" category $V$. Typical examples of ``reasonable" categories include no condition, Abelian or exact categories.

The filtered spaces, $Y_\bullet$, of persistence theory often arise by considering a dataset, fixing a parameter space, and deciding on a scheme which associates a space to each parameter value, e.g., the \v{C}ech or Vietoris--Rips complexes associated to a point cloud and a real number.  In practice, our parameter space, $X$, is a manifold. If our parameter space is equipped with an embedding $X \hookrightarrow \RR$, we are in the setting of one-parameter persistence. If our parameter space is embedded in $\RR^d$, then we are in the setting of multiparameter persistence. Although $d$-dimensional Euclidean space is perhaps the typical embedding target, in general, an embedding into any manifold of dimension $d$ is a setting for $d$-parameter persistence. 

One-parameter persistence is well understood and under reasonable hypotheses, there are complete, discrete, computable invariants.  Multiparameter persistence is more subtle, and it is a charge of the community to find computable, descriptive invariants.  See the original work of Carlsson and Zomorodian \cite{CZmulti} or the more recent survey \cite{BL}.

Inspired by the works of Botnan, Oppermann, and Oudot \cite{BOO}, and Grady and Schenfisch \cite{grady2021zig}, we study the universal additive invariant of persistence modules: their algebraic $K$-theory. In the present, we use the same setup as in \cite{grady2021zig}, whereby persistence modules are defined to be constructible cosheaves over the parameter space, which itself has been stratified by an ``event stratification."  By imposing some mild hypotheses on our stratified parameter spaces, the category of such constructible sheaves is equivalent to the category representations of the (combinatorial) entrance path category of our space. (See \cite{CP} for a nice overview of this correspondence.) This category of representations---a functor category---has well-defined algebraic $K$-theory for nice target categories, e.g., modules over a commutative ring, pointed sets.

Note that, while persistence modules are most typically defined as representations of partially ordered sets (posets), we define stratified parameter spaces in a direction-free way. Roughly speaking, we keep track of ``event times'' as zero strata, but the data of how event times relate to one another is not explicitly stored in the stratification of a parameter space. Instead, we incorporate the poset structure of a persistence module into the assignments of morphisms out of the stratified parameter space; see~\cite{grady2021zig} for further insights into this perspective. One utility of this formalism is that it puts monotone persistence, zig-zag persistence, and their multiparameter generalizations on common footing.

Our main contributions in the present article are to:
\begin{enumerate}
\item Prove that the $K$-theory of multiparameter persistence (grid) modules is additive over the strata of the parameter space. This main result is \thmref{bigOplus}.  As an immediate corollary, we obtain the groups $K_0$ and $K_1$ for persistence modules valued in finite dimensional vector spaces over a field. Our result is a generalization of Theorem 4.1.6 of \cite{grady2021zig} to the multiparameter setting.
\item For the groups $K_0$, compare our results with those of \cite{BOO}.  In particular, while our $K_0$ is not isomorphic to the Grothendieck group of rank-exact persistence modules of Botnan, Oppermann, and Oudot, we do obtain an explicit comparison homomorphism. This comparison is the content of Section \ref{sect:rankexact}.
\end{enumerate}

\subsection*{Acknowledgements} RG is supported by the Simons Foundation under Travel Support/Collaboration 9966728 and AS is supported by the National Science Foundation under NIH/NSF DMS 1664858.  The authors wish to thank Steve Oudot for helpful correspondence as well as Brittany Fasy for thoughtful feedback.

\section{Background and Conventions}

\subsection{Stratified Spaces and Entrance Paths}

\begin{definition}
A {\em stratified topological space} is a triple $(Y \xrightarrow{\phi} \cP)$ consisting of
\begin{itemize}
    \item a paracompact, Hausdorff topological space, $Y$,
    \item a poset $\cP$, equipped with the upward closed topology, and
    \item a continuous map $Y \xrightarrow{\phi} \cP$.
\end{itemize}
\end{definition}

Note that any topological space is stratified by the terminal poset consisting of a singleton set. Moreover, the simplices of a simplicial complex, $K$, come equipped with the structure of a poset, and we call the resulting stratification of $K$ the \emph{face stratification} which we denote by $\Nat(K)$ \cite{stanley1991f}. 

\begin{definition} Given a stratified topological space $\phi : Y \to \cP$, and any $p \in \cP$, the {\em $p$-stratum}, 
$Y_p$, is defined as
\[
Y_p := \phi^{-1}(p)
\]
\end{definition}

\begin{definition}
 Let $X$ be a manifold with corners.  A {\em cubulation} of $X$ is a covering by embedded cubes such that their interiors are disjoint and every nontrivial intersection of cubes consists of their common lower-dimensional face. A {\em cubical manifold} is a manifold (with corners) equipped with a (fixed) cubulation.
\end{definition}

Analogous to simplicial complexes/combinatorial manifolds, cubical manifolds are naturally stratified by the face poset of the cubulation and we denote the resulting stratification by $\Nat(X)$.

\begin{definition}
Let $X$ be a cubical manifold stratified by $\Nat(X)$. The \emph{combinatorial entrance path category}, $\sf{Ent}_\Delta (X)$ has as objects the strata of $X$ and a morphism $\sigma \to \tau$ whenever $\tau$ is a face of $\sigma$.
\end{definition}

\subsection{Grid Modules}

As noted above, we define our persistence modules as representations of the (combinatorial) entrance path category associated to a stratified parameter space.

\begin{definition}\label{def:kofzz}
Let $X$ be a cubical manifold with its face stratification, $\mathsf{Ent}_\Delta (X)$ its combinatorial entrance path category, and   $V$ any category. The {\em category of $V$-valued persistence modules parameterized by $X$}, $\mathsf{pMod}^{V}(X)$, is given by
\[
\mathsf{pMod}^{V}(X):=\mathsf{Fun} (\mathsf{Ent}_\Delta (X),V).
\]
Hence, the \emph{$K$-theory of $V$ valued persistence modules (parametrized by $X$)} is the $K$-theory spectrum (whenever it exists) of the category above: $\KK(\mathsf{pMod}^{V}(X))$.
\end{definition}

In what follows, we are  interested in {\em grid modules}, i.e., those persistence modules parametrized by subspaces of cubulated $\RR^d$.  To that end, we need a few preliminary definitions. 

Given a stratified space $A$ and an embedding $A \hookrightarrow Z$, there is a stratification of $Z$ extending that on $A$ called the {\em connected ambient stratification} and the resulting stratified space is denoted $(Z,A)^{\wwedge}$.  See Definition 2.1.11 of \cite{grady2021zig} for details.

\begin{definition}[Stratified $d$-Parameter Space and its Entrance Path Category]
    \label{def:multii}
    Let $\cI = \{I_i\}_{i=1}^{i=d}$ be a collection of finite subsets of $\R$, so $I_i \subset \R$ is finite.
Define the stratified space
    \[
        (\R^d; \cI) := \prod_{n=1}^d (\R, I_n)^{\wwedge}.
    \]
\end{definition}

    \begin{example}[Stratified Two-Parameter Space and its Stratifying Set]
	   Consider the set of subsets $\cI= \{\{0, 1, \dots, 6 \},\{0,1, \dots, 4\}\}$. These subsets then define a
        stratified space, $(\R^2; \cI)$, shown in \figref{bifiltration_grid} as the large rectangular grid. This stratified space
        has 35 zero strata, 58 one-strata, and 24 two-strata corresponding to vertices, edges, and faces, respectively. Note that this means the stratifying poset of    $(\R^d; \cI)$  then has $35+58+24 = 117$ objects.
    \end{example}

\begin{figure}[h!]
    \centering
    \includegraphics[width=.4\textwidth]{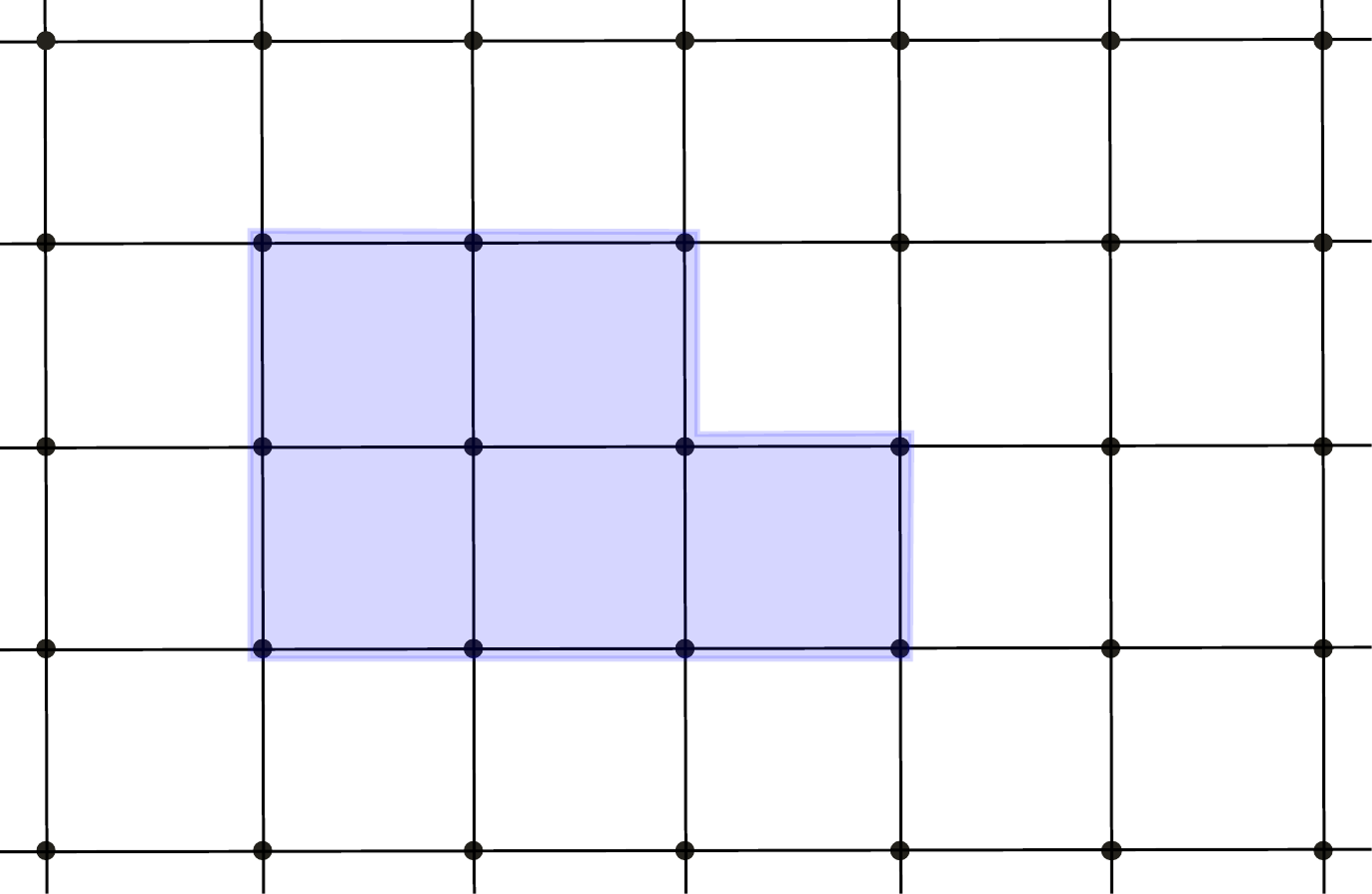}
    \caption{A cubical two-manifold (shaded) appearing as a substratified space of the stratified parameter space $(\R^2; \{\{0, 1, \dots, 6 \},\{0,1, \dots, 4\}\})$.} \label{fig:bifiltration_grid}
\end{figure}

Note that any closed and bounded substratified space of $(\R^d;\cI)$ is naturally cubulated.

\begin{definition}[Cubical Grid Manifold]
 A {\em cubical grid $d$-manifold} is a cubical manifold of dimension $d$ that is embedded as a substratified space of $(\R^d; \cI)$ for some $\cI$.  
\end{definition}

See~\figref{bifiltration_grid} for an instance of a cubical grid manifold.

\subsection{Waldhausen $K$-Theory}

Dan Quillen, justifiably, received much recognition for defining higher algebraic $K$-theory via Abelian and exact categories.  Some twenty years later, Friedhelm Waldhausen found a further generalization of Quillen's setting in his work on the algebraic $K$-theory of spaces \cite{Wald}. Today, this setting is that of Waldhausen categories and exact functors between them. Here we recall, tersely, some key notions leading to Waldhausen's Additivity Theorem. A modern introduction to this material can be found in \cite{fiore} or the encyclopedic \cite{Weibel}.

\begin{definition} Let $\sA, \sE,$ and $\sB$ be Waldhausen categories. A sequence of exact functors
\[
\sA \xrightarrow{i} \sE \xrightarrow{f} \sB
\]
is \emph{exact} if
\begin{enumerate}
\item The composition $f \circ i$ is the zero map to $\sB$;
\item The functor $i$ is fully faithful; and
\item The functor $f$ restricts to an equivalence between $\sE / \sA$ and $\sB$.\footnote{Here,  $\sE/\sA$ is the full subcategory of $\sE$ on objects $e$ such that, for all $a \in \sA$, the hom set $\sE (i(a), e)$ is a point.} 
\end{enumerate}
A sequence, as above, is \emph{split} if there exist exact functors
\[
\sA \xleftarrow {j} \sE \xleftarrow{g} \sB
\]
that are adjoint to $i$ and $f$ and such that the unit of the adjunction, $\mathrm{Id}_\sA \Rightarrow j \circ i$, and the counit of the adjunction, $f \circ g \Rightarrow \mathrm{Id}_\sB$, are natural isomorphisms.
\end{definition}

\begin{definition}\label{defn:standard}
 A split short exact sequence of Waldhausen categories
\begin{center}
\begin{tikzcd}
\sA \ar[r,"i"] &
\sE \ar[r,"f"] \ar[l, bend left, "j"]&
\sB \ar[l, bend left, "g"]
\end{tikzcd}
\end{center}
is \emph{standard} if 
\begin{enumerate}
\item For each $e \in \sE$, the component of the counit, $(i \circ j)(e) \to e$, is a cofibration;
\item For each cofibration $e \hookrightarrow e'$ in $\sE$, the induced map
\[
e \amalg_{(i \circ j)(e)} (i \circ j) (e') \to e'
\]
is a cofibration; and
\item If $a \to a' \to 0$ is a cofiber sequence in $\sA$, then the first map is an isomorphism.
\end{enumerate}
\end{definition}

The following is one of the fundamental theorems of algebraic K-theory. It is known as {\it Waldhausen Additivity}.

\begin{theorem}\label{thm:additivity}
Let
\begin{center}
\begin{tikzcd}
\sA \ar[r,"i"] &
\sE \ar[r,"f"] \ar[l, bend left, "j"]&
\sB \ar[l, bend left, "g"]
\end{tikzcd}
\end{center}
be a standard split SES of Waldhausen categories. Then the functors $i$ and $g$ induce an equivalence of spectra
\[
\KK(i) \vee \KK(g) \colon \KK(\sA) \vee \KK(\sB) \xrightarrow{\sim} \KK(\sE).
\]
\end{theorem}

\section{$K$-Theory of Grid Modules}

In this section we prove our main theorem, which is the multiparameter analog of Theorem 4.1.6 of \cite{grady2021zig}.

It is standard that the category of finitely generated, projective modules for a commutative ring defines a Waldhausen category. So too does the category of functors from a small category into this category of modules.  (Some details are provided in Appendix A of \cite{grady2021zig}.)
\begin{lemma}\label{lem:key}
    Let $R$ be a commutative ring, $\cM$ the associated Waldhausen category of
    finitely generated projective modules. Furthermore, let $X$ be a cubical manifold
    and let $A$ denote a closed
    sub-stratified space of $X$.
    Then the following sequence is split short exact sequence of
    Waldhausen categories
    \begin{center}
        \begin{tikzcd}
            \mathsf{Fun} (\mathsf{Ent}_\Delta (X \setminus A),\cM)\ar[r,
            "j_\ast"] &
            \mathsf{Fun} (\mathsf{Ent}_\Delta (X),\cM) \ar[r, "i^\ast"]
            \ar[l, bend left, "j^\ast"] &
            \mathsf{Fun} (\mathsf{Ent}_\Delta (A),\cM) 
            \ar[l, bend left, "i_\ast"],
        \end{tikzcd}
    \end{center}
     where $i \colon A \hookrightarrow X$ and $j \colon X \setminus A
     \hookrightarrow X$ are the inclusion maps. Moreover, this sequence is standard.
\end{lemma}

The argument for the preceding lemma is, {\em mutatis mutandis}, as for Lemma 4.1.1 of \cite{grady2021zig}.

Since the sequence in \lemref{key} is a standard split short exact sequence of Waldhausen
categories, Waldhausen Additivity (\thmref{additivity}) immediately gives us the
following key corollary, which is the main tool we will use to in our argument
to``break apart and
glue together'' modules from submodules.

\begin{corollary}[$K$-Theory is Additive Over Sub-Modules]\label{cor:key}
    Let $R$ be a commutative ring,~$\cM$ the associated Waldhausen category of
    finitely generated projective modules. Furthermore, let $X$ be a cubical manifold and 
    let $A$ denote a closed
    sub-stratified space of $X$.
    Then we have an equivalence of spectra
    \[
        \KK(\mathsf{pMod^{\cM}} (X)) \cong \KK(\mathsf{pMod^{\cM}} (X \setminus
        A)) \vee
        \KK(\mathsf{pMod^{\cM}} (A)).
    \]
\end{corollary}

The following is a main result of 
(\cite[Lemma 4.1.5]{grady2021zig}), and serves as a
base case for our induction into multi-parameter modules.

\begin{lemma}\label{lem:stratadd}
Let $X$ be a cubical one-manifold with a finite set of strata.
 There is an  equivalence of
spectra
\[
\KK(\mathsf{pMod^{\cM}} (X)) \cong \bigvee_{x_0 \in X_0} \KK(\mathsf{pMod^{\cM}}
(x_0)) \vee \bigvee_{x_1 \in X_1} \KK(\mathsf{pMod^{\cM}} (x_1)).
\]
where $X_i$ is the set of $i$-strata of $X$.
\end{lemma}

In proving the preceding result, the authors inducted on the number of zero-strata
of $X$. In our multi-parameter setting we  induct on an analogous notion: {\em height}.

\begin{definition}
    Let $\cI = \{I_i\}_{i=1}^{i=d}$ be a collection of finite subsets of $\R$. We say the \emph{height} of $\cI$ is the maximum
    number of objects in any
    $I_n \in \cI$.
    \end{definition}
    
\begin{definition}
    Given a cubical grid $d$-manifold $X$, we say the \emph{height} of $X$ is
    the minimum height over all collections $\cI$ such that $X \subseteq (\R^d; \cI)$.
    If the height of $X$ equals the number of objects in $I_n$,
    we say that the $n$th parameter \emph{realizes the height of $X$}.
\end{definition}

    It is clear from the definition that in the one-parameter case, the height of $X$ is indeed just the
    number of zero-strata of $X$. In the general $d$-parameter case, height is a
    measure of the longest axis-aligned ``slice,'' although note that the height of $X$
    may be realized in more than one parameter. 

We are now ready to state our main result, the $K$-theory of multi-parameter
 grid modules. The proof uses a double induction on the number of
parameters and on the height of the module.

\begin{theorem}[$K$-Theory of Multi-Parameter Zig-Zag
    Modules]
    \label{thm:multistratadd}
Let $X$ be a cubical grid~$d$-manifold with a finite number of strata.
  There is an  equivalence of spectra
    \[
        \KK(\mathsf{pMod^{\cM}} (X)) \cong \bigvee_{x_0 \in X_0} \KK(\cM) \vee \bigvee_{x_1 \in X_1} \KK(\cM) \vee
        \ldots \vee \bigvee_{x_d \in X_d} \KK(\cM)
    \]
    where $X_i$ is the set of $i$-strata of $X$.
\end{theorem}

\begin{proof}
    We proceed by double induction; first on $h$, the height of the persistence
    module, and then on $d$, the number of parameters.
    As a base case, we observe that \lemref{stratadd} asserts the statment
    holds for $d=1$ and all $h$. 
    Suppose that there exist $d_*, h_* \in \NN$ so that the claim holds for all
    $1 \leq d \leq d_*$ and $1 < h \leq h_*$.
    \footnote{Here, we make the inequality $h>1$ strict to avoid
    tautologies; a
    cubical manifold with height one is simply a point.}

    \textbf{Induction on Height:}
    First, we induct on the height of $X$.  Suppose that $X$ is
    $d_*$-dimensional and has height $h_* + 1$. For simplicity, we first
    consider the case that only one parameter realizes the height $h_* +1$.  and
    suppose this parameter is \emph{not} the $i$th parameter for some $i \in [1,d]$.
    Then, for some $m \in I_i$, consider the closed $(d_*-1)$-parameter
    sub-stratified space $A$ corresponding to the poset $I_1 \times \ldots
    \times I_{i-1} \times m \times I_{i+1} \times \ldots I_d$.
    Conceptually,~$A$ corresponds to a level-set of $X$ at height $m$, slicing
    through the parameter that realizes the maximum height.  Note that $X
    \setminus A$ is generally two disconnected $d_*$-dimensional spaces (or one
    $d_*$-dimensional space in the specific case that $A$ is on the boundary of
    $X$), each with height no more than $h_*$.  Furthermore, note that $A$ has
    height no more than $h_*$. By the inductive hypothesis, the claim holds for
    the connected components of $X \setminus A$ as well as for~$A$, so by
    \corref{key}, we see that the claim holds for all of $X$.

    Next, consider the general case, i.e., the case that any number of parameters realize the height $h_*+1$.
    We describe the process of dividing $X$ into connected components each with
    height no more than $h_*$ algorithmically, using a stack data structure.
    Recall that, just like a stack of plates, a stack utilizes a ``first on,
    first off'' organization, where the element that was most recently pushed
    onto the stack is the element available to be popped off.
    Specifically, we use a stack, $T$, of ``tall''
    connected components of this division that have height $h_*+1$, initialized to
    $T = X$. We keep track of a list $S$ of ``short'' connected components that
    have height less than $h_*+1$, initialized as empty.  The procedure is as
    follows. First, we pop $X_i \in T$. Then, we choose a closed $(\dim(X_i)-1)$-parameter
    sub-stratified space,~$A_i \subset X_i$, that is perpendicular to some parameter of
    $X_i$ that realizes height $h_*+1$. Note then that the connected components
    of $X_i \setminus A_i$ and
    $A_i$ may still have height $h_*+1$, but in one fewer parameter than $X_i$
    had height $h_*+1$.  We push the connected components of $X_i
    \setminus A_i$ and $A_i$ that have height $h_*+1$ back on the stack $T$ and
    move any
    connected components of this division with height less than $h_*+1$ to our
    list $S$. Since each processed element of $T$ has height $h_*+1$ in one
    fewer parameter direction
    than before it was processed, we eventually have $T$ empty and $S$ a
    division of $X$ into cubical grid manifolds for modules each with height less than
    $h_*+1$.
    Noting that each time a sub-stratified space was removed, it was a closed
    subspace of the cubical manifold containing it, we ``glue'' back the pieces
    using \corref{key}, eventually showing the $K$-theory for persistence
    modules over all of $X$ is as claimed.


    \textbf{Induction on Number of Parameters:} Next, we induct on the number of
    parameters. From our preceeding inductive argument, it is sufficeint to let $X'$ be a cubical grid $(d_*
    +1)$-manifold with height two. This means $X'$ is a $(d_* +1)$-cube.
    Let~$\cube$ denote the $(d_*+1)$-stratum of this cube (i.e., the interior).
    Since $X' \setminus \cube$ is a closed substratified space of $X'$, we know
    by \corref{key} that we have an equivalence of spectra
    \begin{align}
        \KK(\mathsf{pMod^{\cM}} (X')) \cong \KK(\mathsf{pMod^{\cM}} (\cube))
        \vee \KK(\mathsf{pMod^{\cM}}
        (X' \setminus \cube)).
        \label{eqn:removeguts}
    \end{align}
    (Note the slight flip of roles and notation from \corref{key}; our closed
    subspace is $A = X' \setminus \cube$ and thus, $X' \setminus A = X' \setminus
    (X' \setminus \cube) = \cube$).

    It remains, then, to compute the $K$-theory of persistence modules over $X'
    \setminus \cube$. The space $X' \setminus \cube$ has the geometric structure of
    the boundary of a $(d_*+1)$-cube, or, equivalently, a cube with no interior.
    By~\cite[Theorem 3]{cubes}, every unfolding of a cube with empty interior
    along a connected collection of codimension-two faces will not self-overlap
    and will lie on grid points of space in one dimension lower (i.e., every
    \emph{ridge unfolding} of a finite cube will produce a
    \emph{net})\footnote{In~\cite{cubes}, the word ``cube'' is taken
    to mean a cube with empty
    interior.}. 

	As discussed in~\cite{cubes}, ridge unfoldings correspond to trees
     in the cube's dual. A path in this dual
    corresponds to an unfolding where codimension-two faces are only pairwise
    adjacent to one other. This pairwise adjacency is along a common
    codimension-three face, so that the collection of codimension-two faces
    corresponds to a submanifold (with boundary) of the cube.
    Choose some such unfolding of the $(d_*
    +1)$-cube, denoting the corresponding closed and connected collection of
    codimension-two faces (i.e., $(d_*-1)$-faces) by $U$.  Since the unfolding
    forms a net, it embeds into some $(\R^{d_*}; \cI)$, so 
    $(X'\setminus \cube) \setminus U$ may be no more than a
    cubical grid $d_*$-manifold. Since $U$ corresponds to a path in the dual graph of the
    cube, $U$ is a cubical grid $(d_* -2)$-manifold. 
    (see \figref{cubeunfolding}). Thus, by the
    inductive hypothesis, the claim holds for $(X' \setminus \cube) \setminus U$
    as well as for $U$.  Then by \corref{key}, we see that the claim holds for
    all of $X' \setminus \cube$.  Combining this result with the
    equivalence in \eqnref{removeguts}, we have shown the desired result holds
    for all of $X$.

    Next, we observe that since we have shown that the claim holding for a cubical
    grid $d$-manifold implies the claim holds for a cubical
    $(d+1)$-manifold with height two, and since we have furthermore shown the
    claim holds for cubical grid manifolds of any height, we have shown the
    desired result in generality.
    
    Finally, we identify the $K$-theory of components of the stratification, i.e., we
    identify $\KK(\mathsf{pMod^{\cM}} (x_i))$ for $x_i \in X_i$ and $i \in
    \{1, \ldots, d\}$. By \defref{kofzz}, we have $\KK(\mathsf{pMod^{\cM}} (x_i)) =
\KK(\mathsf{Fun} (\mathsf{Ent}_\Delta (x_i),\cM))$. Since
$\mathsf{Ent}_\Delta(x_i)$ is the terminal category (a single object and an
identity morphism), $\mathsf{Fun} (\mathsf{Ent}_\Delta (x_i),\cM)$ is
isomorphic to the category of $\cM$ itself. Thus,
$\KK(\mathsf{pMod^{\cM}} (x_i)) = \KK(\cM)$. 
\end{proof}

\begin{figure}
    \centering
    \includegraphics[scale=.8]{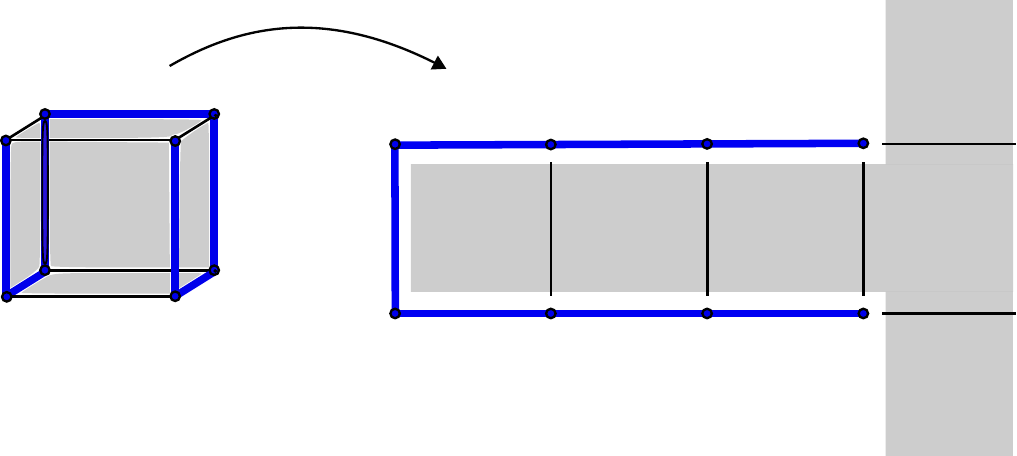}
    \caption{The unfolding of a three-dimensional cube $C$ with height two in
    each parameter (an instance of the cube $X' \setminus \cube$ discussed in
    the induction on height in the proof of
    \thmref{multistratadd}). The thick
    blue submodule $U$ is a closed connected collection of codimension-two faces
    of the cube along which we can unfold (right). Since the unfolding (left) is a net,
    connected components of both the $C \setminus U$ and $U$ cannot be more than
    two-parameter modules.}\label{fig:cubeunfolding}
\end{figure}

We end this section by discussing how \thmref{multistratadd} translates to the
specific case of $\Vect$-valued multi-parameter persistence modules. 

\begin{theorem}[$K$-theory of $\Vect$-Valued Multi-Parameter Modules]\label{thm:bigOplus}
Let $X$ be a cubical grid $d$-manifold with a finite number of strata.
There is an  equivalence
of spectra
\[
\KK(\mathsf{pMod^{\Vect_\FF}} (X)) \cong \bigvee_{X_0} \KK (\FF) \vee
\bigvee_{X_1} \KK (\FF) \vee \ldots \vee \bigvee_{X_d} \KK (\FF)
\]
where $X_i$ is the set of $i$-strata of $X$ and $\KK (\FF)$ denotes the
$K$-theory spectrum of the field $\FF$.
\end{theorem}

\begin{proof}
Now, the category of
finite dimensional vector spaces over $\FF$ is exactly the category of finitely
generated projective modules over $\FF$ (considered as a ring). Hence,
$\KK(\Vect_\FF)$ is just the algebraic $K$-theory of $\FF$.

Thus, we have shown the $K$-theory of each strata is a copy of $K(\FF)$. We
    know by \thmref{multistratadd} that $\KK(\mathsf{pMod^{\Vect_\FF}} (X))$ is
additive over strata, so the result follows.
\end{proof}

The first two $K$-groups of a field are well known.  The following isomorphisms are induced by the dimension and determinant maps, respectively. 

\begin{corollary}
For $X$, a cubical grid $d$-manifold with a finite number of strata, we have
\[
K_0(\mathsf{pMod^{\Vect_\FF}} (X)) \cong \bigoplus_{X_0} \ZZ \oplus
    \bigoplus_{X_1} \ZZ \oplus \ldots \oplus \bigoplus_{X_d} \ZZ
\]
and
\[
K_1(\mathsf{pMod^{\Vect_\FF}} (X)) \cong \bigoplus_{X_0} \FF^\times \oplus
\bigoplus_{X_1} \FF^\times \oplus \ldots \oplus \bigoplus_{X_d} \FF^\times
\]
where $X_i$ is the set of $i$-strata of $X$ and $\FF^\times$ is the group of
units of $\FF$.
\end{corollary}

See Chapter IV of \cite{Weibel} for an in-depth description of the higher $K$-theory of fields.

\section{Connections: Euler Manifolds and Rank Exact $K$-Theory}
\label{sec:connections}

As observed by Grady and Schenfisch in the one-dimensional case \cite{grady2021zig}, given a persistence module, $\cF$, over a parameter space, $X$, the class of $\cF$ in $K_0$ is the Euler curve of the persistence module.  The same conclusion holds in the multiparameter setting with the exception that we no longer consider the Euler curve, but rather the Euler surface or manifold depending on the number of parameters. Indeed, the description of $K_0$ in terms of constructible functions goes back to Kashiwara and Schapira \cite{KS}, and the construction of their isomorphism uses a local Euler index.

\subsection{Rank Exact $K$-Theory}\label{sect:rankexact}

Finally, we briefly discuss the relationship between our present results and the recent work of Botnan, Oppermann, and Oudot on Grothendieck groups, $K_0$, in multiparameter persistence via rank-exact structures \cite{BOO}.

Let $\FF$ be a field and $\cP$ be an arbitrary poset. Let $\mathsf{Rep}(\cP)$ denote the category of functors~$\cP \to \Vect_\FF$ with finite total rank.  Note, we use rank, $\mathop{\rm Rk}$, instead of dimension (over $\FF$), as there are many other choices of target category to which the work of \cite{BOO} applies. Of course, in our simplified setting $\mathop{\rm Rk}=\dim_\FF$.

\begin{definition}[Definition 4.1 \cite{BOO}]
A short exact sequence $0 \to \sF \to \sG \to \sH \to 0$ in $\mathsf{Rep}(\cP)$ is \em{rank-exact} if $\mathop{\rm Rk} \sG = \mathop{\rm Rk} \sF + \mathop{\rm Rk} \sH$.
\end{definition}

Theorem 4.4 of \cite{BOO} states that $\mathsf{Rep}(\cP)$ equipped with rank-exact short exact sequences is an exact category, which we denote $\mathsf{Rep}(\cP)_{\mathop{\rm Rk}}$.  Hence, $K_0 (\mathsf{Rep}(\cP)_{\mathop{\rm Rk}})$ is well-defined.  The higher $K$-groups also exist, but following the authors of {\em loc.~cit.}~we restrict our discussion to $K_0$.

For $\cP$ a poset, let $\mathop{\rm Seg}(\cP)$ be the collection of \emph{segments}, i.e., pairs defining the partial order on the underlying set of $\cP$, i.e., 
\[
\mathop{\rm Seg}(\cP) = \{ (p,p')\in \cP \times \cP : p \le p'\}.
\]

\begin{theorem}[Theorem 4.10 \cite{BOO}]
Let $\cP$ be a finite poset.  The rank map induces an isomorphism
\[
\mathop{\rm Rk} \colon K_0 (\mathsf{Rep}(\cP)_{\mathop{\rm Rk}}) \xrightarrow{\; \; \approx \; \; } \ZZ^{\mathop{\rm Seg}(\cP)}.
\]
\end{theorem}

Consider the linear order, $[2]=\{0\le1\le2\}$.  Following \cite{grady2021zig}, the poset $[2]$ defines a cubical one-manifold with three zero-strata and two one-strata, which we denote by $\cX([2])$.  By the additivity results of Grady and Schenfisch, or the one-parameter specialization of the results above, $K_0(\mathsf{pMod^{\Vect}} (\cX([2]))) \cong \ZZ^5$.   However, note that $\lvert \mathop{\rm Seg}([2]) \rvert = 6,$ so $K_0 (\mathsf{Rep}([2])_{\mathop{\rm Rk}}) \cong \ZZ^6$.  Of course, $5 \neq 6$, so we seek an explanation/explicit comparison.



Let $\cP$ be a finite poset. The {\em order complex}, $\le^{\cP}_\bullet$, is the abstract simplicial complex whose faces are chains in $\cP$.  Let $\lvert \cP \rvert := \lvert \le^{\cP}_\bullet \rvert$ denote the geometric realization of $\cP$, which is a (geometric) simplicial complex. 
It is a standard exercise that $\lvert [n] \rvert \cong \Delta^n$, i.e., the realization of the linear order $\{0\le1 \le \dotsb \le n\}$ is the standard $n$-simplex.  Returning to the previous paragraph, note that the cubical one-manifold $\cX([2])$ is a subcomplex of $\Delta^2=\lvert [2] \rvert$; it sits inside as the {\em spine} of the 2-simplex. Note further that there is a natural bijection between $\mathop{\rm Seg}([2])$ and the set of simplices of the one skeleton of $\Delta^2$.  Indeed, this bijection is determined by considering the chains of length at most one in $[2]$, i.e., the one skeleton of the order complex, $\mathop{{\rm sk}_1} \le^{[2]}_\bullet$.  We conclude that the inclusion 
\[
\cX([2])=\mathop{\rm spine} \Delta^2 \hookrightarrow \mathop{{\rm sk}_1} \Delta^2
\]
induces a projection map $K_0 (\mathsf{Rep}([2])_{\mathop{\rm Rk}}) \cong \ZZ^6 \to \ZZ^5 \cong K_0(\mathsf{pMod^{\Vect}} (\cX([2])))$. This argument immediately generalizes to any finite linear order, and we have the following.

\begin{prop}
For each $n \in \NN$, the inclusion of the spine into the one skeleton of the $n$-simplex induces a projection
\[
K_0 (\mathsf{Rep}([n])_{\mathop{\rm Rk}}) \cong \ZZ^{\frac{n(n+1)}{2}} \to \ZZ^{2n-1} \cong K_0(\mathsf{pMod^{\Vect}} (\cX([n]))).
\]
\end{prop}

Similarly, we can analyze the case of a grid as well. Let $\cG = [n_1]\times [n_2] \times \dotsb \times [n_d]$ be a finite grid equipped with the (categorical) product poset structure.  Again, we have a natural bijection $\mathop{\rm Seg}(\cG) \cong \mathop{{\rm sk}_1} \le^{\cG}_\bullet$. Next, note that strata of $\cX(\cG)$ are indexed by the Cartesian product of sets $\mathop{\rm spine} \lvert [n_1] \rvert \times \mathop{\rm spine} \lvert [n_2] \rvert \times \dotsb \times \mathop{\rm spine} \lvert [n_d] \rvert$, where $\cX(\cG)$ is the $d$-dimensional cubical grid manifold associated to $\cG$. Now, for each $k$ and each edge or vertex $\alpha \in \mathop{\rm spine} \lvert [n_k] \rvert$, there is a {\em source} $s(\alpha) \in \cG$ and a {\em target} $t(\alpha) \in \cG$.  There is an injection of sets
\[ 
\iota \colon \mathop{\rm spine} \lvert [n_1] \rvert \times \mathop{\rm spine} \lvert [n_2] \rvert \times \dotsb \times \mathop{\rm spine} \lvert [n_d] \rvert
 \hookrightarrow \mathop{\rm Seg}(\cG)
\]
given by
\[
\iota(\alpha_1, \alpha_2 , \dotsc , \alpha_d) = ((s(\alpha_1), s(\alpha_2), \dotsc, s(\alpha_d)), (t(\alpha_1), t(\alpha_2), \dotsc, t(\alpha_d))).
\]
Geometrically, the map $\iota$ exhibits $\cX(\cG)$ as a coarsening of a subcomplex of $\lvert \cG \rvert$. As before we obtain a projection map at the level of $K$-groups.

\begin{prop}
Let $\cG$ be a finite grid.  The map $\iota$ induces a projection
\[
K_0 (\mathsf{Rep}(\cG)_{\mathop{\rm Rk}}) \to K_0(\mathsf{pMod^{\Vect}} (\cX(\cG)).
\]
\end{prop}

Thus, when $\cP$ is of the form $[n]$, or more generally is a finite grid, we have defined a projection map that allows for direct comparison between the group $K_0 (\mathsf{Rep}(\cP)_{\mathop{\rm Rk}})$ of Botnan, Oppermann, and Oudot, and the group $K_0(\mathsf{pMod^{\Vect}} (\cX(\cP))$ of the present work.

\bibliographystyle{amsalpha}
\bibliography{refs}
\end{document}